\documentclass[reqno,oneside]{amsart}

\usepackage{hyperref}
\usepackage{amssymb,amsmath,amssymb,amsfonts,amsthm}
\usepackage{cases}
\usepackage{empheq}
\usepackage{url}
\usepackage{geometry}
\usepackage{cite,mathtools}

\numberwithin{equation}{section}

\theoremstyle{plain}
\newtheorem{theorem}{Theorem}[section]

\newtheorem{lemma}[theorem]{Lemma}
\newtheorem{proposition}[theorem]{Proposition}
\theoremstyle{definition}

\theoremstyle{remark}
\newtheorem{remark}[theorem]{Remark}
\newtheorem{example}[theorem]{Example}
\numberwithin{equation}{section}
\newcommand{\ep}{\varepsilon}

\allowdisplaybreaks[4]

\def\rd{\mathrm d}

\def\e{\mathrm e}

\def\Om{\Omega}
\def\al{\alpha}

\def\de{\delta}
\def\ep{\epsilon}
\def\ve{\varepsilon}

\def\la{\lambda}
\def\si{\sigma}

\def\om{\omega}
\def\f{\frac}
\def\nb{\nabla}
\def\ov{\overline}
\def\pa{\partial}
\def\wt{\widetilde}

\title[Energy decay of viscoelastic wave problem]
{Energy decay of a viscoelastic wave equation with\\
variable exponent logarithmic nonlinearity\\
and weak damping}

\author[Qingqing Peng]{Qingqing Peng$^{1}$}
\thanks{$^1$School of Mathematics and Statistics \& Hubei Key Laboratory of Engineering Modeling and Scientific Computing, Huazhong University of Science and Technology, Wuhan 430074, China.}

\author[Yikan Liu]{Yikan Liu$^{2,*}$}
\thanks{$^2$ Department of Mathematics, Kyoto University, Kitashirakawa-Oiwakecho, Sakyo-ku, Kyoto 606-8502, Japan.}
\thanks{$^*$Corresponding author.  E-mail: {\tt liu.yikan.8z@kyoto-u.ac.jp}}

\keywords{Viscoelastic wave equation, variable exponent logarithmic nonlinear, weakly damping, exponential and polynomial decay.}

\begin{document}

\begin{abstract}
In this paper, we investigate the energy decay of the solution to a viscoelastic wave equation with variable exponents logarithmic nonlinearity and weak damping in a bounded domain. We establish an explicit general decay result under mild conditions on the relaxation function $g$. Furthermore, under the general assumption $g'(t)\leq-\zeta(t)G(g(t))$ with some suitably given $\zeta$ and $G$, we derive a refined decay estimate improving existing results. In particular, uniform exponential and polynomial decay rates are obtained under a further special situation $g'(t)\leq-\xi(t)g^q(t)$ with $1\leq q<2$, extending earlier studies that were restricted to the case $1\leq q<\frac{3}{2}$.
\end{abstract}

\maketitle


\section{Introduction}

Let $\Omega\subset\mathbb{R}^n$ ($n\geq3$) be a bounded domain with a smooth boundary $\pa\Om$. This article is concerned with the initial-boundary value problem for a nonlinear viscoelastic wave equation
\begin{equation}\label{1.1}
\left\{\begin{alignedat}{2}
& u_{tt}-\triangle u+\int_0^t g(t-s)\triangle u(s)\,\rd s+u_t=\alpha\vert u\vert^{p(x)-2}u\log\vert u\vert,
& \quad & (x,t)\in\Omega\times(0,\infty),\\
& u(x,0)=u_0(x),\ u_t(x,0)=u_1(x), & \quad & x\in\Omega,\\
& u(x,t)=0, & \quad & (x,t)\in\pa\Om\times(0,\infty),
\end{alignedat}\right.
\end{equation}
where $\alpha>0$ is a constant. Here we assume that the exponent $p$ in the nonlinear term satisfies
\begin{equation}\label{1.2}
p_1:=\mathop{\mathrm{ess}\inf}_{x\in\Om}p(x)>2,\quad p_2:=\mathop{\mathrm{ess}\sup}_{x\in\Om}p(x)<\frac{2(n-1)}{n-2}.
\end{equation}
Moreover, $p$ is further assumed to be log-H\"older continuous, i.e., there exist a constant $A>0$ such that for a.e.\! $x,y\in\Om$ satisfying $\vert x-y\vert<1$, there holds
\begin{equation}\label{1.3}
\vert p(x)-p(y)\vert\leq-\frac{A}{\log\vert x-y\vert}.
\end{equation}

Wave equations with logarithmic nonlinearity have attracted significant attention in recent years, and numerous studies have explored the dynamics of problem \eqref{1.1} in the absence of the memory effect (i.e., when $g=0$). The logarithmic wave equation was first introduced by Bialynicki-Birula and Mycielski in \cite{17,18}, who demonstrated the existence of stable and localized solutions in one spatial dimension. Subsequently, Cazenave and Haraux \cite{19} established the well-posedness of the Cauchy problem in three spatial dimensions. Bartkowski and G\'orka \cite{20} studied classical and weak solutions to the one-dimensional Cauchy problem, while G\'orka \cite{21} proved the global existence of weak solutions for initial data $(u_0,u_1)\in H_0^1(\Omega)\times L^2(\Omega) $ using compactness arguments. More recently, Di \cite{4} applied the potential well method to establish global existence for the wave equation with logarithmic nonlinearity and derived exponential or polynomial decay by constructing an appropriate Lyapunov functional. Moreover, blow-up in the unstable set was also established. For further results on stability and blow-up of systems with logarithmic nonlinearities, we refer the reader to \cite{1,14,qing,qing2,PL26} and the references therein.

On the other hand, under certain assumptions on the kernel function $g$ in the memory term, several decay and blow-up results have been established in the literature. In \cite{22}, the authors studied the stability of the initial-boundary value problem for a quasilinear viscoelastic equation. They obtained polynomial decay under mild conditions on $g$, and further investigated both polynomial and exponential decay under a more general condition on $g$. In \cite{12}, energy decay results for the viscoelastic problem were derived by constructing an appropriate Lyapunov functional under suitable assumptions on $g$. Additionally, the authors in \cite{15} considered blow-up phenomena for the viscoelastic wave problem, while those in \cite{2} analyzed the blow-up of solutions for a viscoelastic problem with variable exponents. Later, Liao \cite{23} discussed energy decay rates for solutions to a viscoelastic wave equation with variable exponents and weak damping. Further related works can be found e.g.\! in \cite{25,3,24,uk,q2}.

Inspired by the aforementioned studies, in this article we investigate the stability and blow-up behavior of problem \eqref{1.1}. This manuscript contains three main contributions. First, in proving the energy decay, we neither require the initial energy to be smaller than the depth of the potential well $d$, nor restrict the analysis to a stable set as required in previous works involving logarithmic nonlinearity, which relaxes the key structural assumptions. Second, the conditions imposed on the kernel $g$ are more general than those considered in earlier studies on systems with variable exponents. Third, in Theorem \ref{theorem3.1}, our result holds for $1\leq q<2$, whereas existing literature only covered the case of $1\leq q<\frac{3}{2}$. Indeed, our findings extend and generalize previous results, particularly as those in \cite{25,22,24,qing}.

The organization of this paper is as follows. In Section \ref{sec2}, we establish the global existence of solutions and present several auxiliary lemmas. Then Section \ref{sec3} and \ref{secc3} are devoted to the stability analysis of problem \eqref{1.1}, achieved through the construction of an appropriate Lyapunov functional and the application of multiplier methods.


\section{Preliminaries and some lemmas}\label{sec2}

In this section, we introduce some notations, basic definitions, essential lemmas, and function spaces that will be used in stating and proving our main results.

Throughout, by $\Vert\cdot\Vert_k$ we denote the norm of the Lebesgue space $L^k(\Omega)$ for $1\leq k\le\infty$, and by $(\,\cdot\,,\,\cdot\,)$ the inner product of $L^2(\Omega)$. In order to study problem \eqref{1.1}, we start with recalling the Orlicz-Sobolev-type Banach spaces $L^{q(x)}(\Omega)$ defined as (see \cite{8,9})
$$
L^{p(x)}(\Omega):=\left\{f:\mbox{a measurable real-valued function in }\Om\mid\int_{\Omega}\vert f\vert^p\,\rd x<\infty\right\},
$$
where $p\in L^\infty(\Om)$ satisfies \eqref{1.2}. The norm of $L^{p(x)}(\Omega)$ is given by
$$
\Vert f\Vert_{p(x)}:=\inf\left\{\lambda>0\mid\int_{\Omega}\left\vert\frac{f}{\lambda}\right\vert^p\rd x\leq1\right\},
$$
and it is readily seen that
\begin{equation}\label{2.1}
\min\left\{\Vert f\Vert_{p(x)}^{p_1},\Vert f\Vert_{p(x)}^{p_2}\right\}\leq\int_{\Omega}\vert f\vert^p\,\rd x\leq\max\left\{\Vert f\Vert_{p(x)}^{p_1},\Vert f\Vert_{p(x)}^{p_2}\right\}.
\end{equation}

We collect several basic facts about the function spaces in the following lemmas.

\begin{lemma}[see \cite{10}]\label{lemma3}
Let $k$ be a constant satisfying $2\leq k\leq 2_*:=\frac{2n}{n-2}$ with $n\geq3$. Then there exists an optimal constant depending only on $k$ such that
$$
\Vert v\Vert_k^k\leq B_k\Vert\nabla v\Vert_2^k,\quad\forall\,v\in H_0^1(\Omega).
$$
\end{lemma}

\begin{lemma}[see \cite{8,9}]
Let $p,q\in C(\ov\Omega)$ satisfy $1<p\le q$ on $\ov\Omega$. Then the embedding $L^{q(x)}(\Omega)\hookrightarrow L^{p(x)}(\Omega)$ is continuous and its operator norm does not exceed $\vert\Omega\vert+1$.
\end{lemma}

\begin{lemma}[see \cite{7}]\label{lemma2}
Let $p$ satisfy {\rm\eqref{1.2}--\eqref{1.3}}. Then the embedding $H_0^1(\Omega)\hookrightarrow L^{p(x)}(\Omega)$ is continuous and compact.
\end{lemma}

Next, we fix the basic assumption on the kernel function $g$ in the memory term.\medskip

{\bf(A1)} The function $g\in C^1([0,\infty);(0,\infty))$ is non-increasing and satisfies
\[
g(0)>0,\quad\ell:=1-\int_0^\infty g(s)\,\rd s>0.
\]

In the sequel, we denote
\begin{equation}\label{eq-IM}
I(t):=\int_t^\infty g(s)\,\rd s,\quad K_{\delta}(s):=-\frac{g'(s)}{g(s)}+\delta,\quad M(\delta):=\int_0^\infty\frac{g(s)}{K_{\delta}(s)}\,\rd s,
\end{equation}
where $\delta\in(0,1)$ is a constant. To state our results, we define the energy functional associated with problem \eqref{1.1} as
\begin{align}
E(t) & :=\frac{1}{2}\Vert u_t(t)\Vert_2^2+\frac{1}{2}\left(1-\int_0^t g(s)\,\rd s\right)\Vert\nabla u(t)\Vert_2^2+\frac{1}{2}(g\circ\nabla u)(t)\nonumber\\
& \quad\:\,-\alpha\int_{\Omega}\frac{|u(t)|^p\log|u(t)|}p\,\rd x+\alpha\int_{\Omega}\frac{|u(t)|^p}{p^2}\,\rd x,\label{2.4}
\end{align}
where
\[
(g\circ\nabla u)(t):=\int_0^{t}g(t-s)\Vert\nabla u(t)-\nabla u(s)\Vert_2^2\,\rd s.
\]
By directly differentiating \eqref{2.4} and using \eqref{1.1}, it is straightforward to verify that
\begin{align}
E'(t) & =\frac{1}{2}(g'\circ\nabla u)(t)-\frac{1}{2}g(t)\Vert\nabla u(t)\Vert_2^2-\Vert u_t(t)\Vert_2^{2}\nonumber\\
& \le\frac{1}{2}(g'\circ\nabla u)(t)-\Vert u_t(t)\Vert_2^{2}\le0.\label{2.5}
\end{align}

\begin{lemma}\label{lemma 2.6}
Let assumption {\rm(A1)} hold, $p$ satisfy {\rm\eqref{1.2}--\eqref{1.3}} and $\mu>0$ be a constant satisfying $p_2+\mu<2_*$. Let $B_{p_2+\mu}$ be the optimal constant in Lemma $\ref{lemma3}$ and define
$$
B:=B_{p_2+\mu}\ell^{-\frac{p_2+\mu}{2}},\quad R(\lambda):=\frac{1}{2}\lambda^2-\frac{\alpha B}{\e\mu p_1}\lambda^{p_2+\mu},\quad\lambda(t):=\left\{\ell\Vert\nabla u(t)\Vert_2^2+(g\circ\nabla u)(t)\right\}^{\frac{1}{2}}.
$$
Then there holds $E(t)\geq R(\lambda(t))$.
\end{lemma}

\begin{proof}
Splitting the domain $\Om$ into
\begin{equation}\label{eq-split}
\Omega_1:=\{v\in H_{0}^1(\Omega)\mid|v|<1\},\quad\Omega_2:=\{v\in H_{0}^1(\Omega)\mid|v|\geq1\},
\end{equation}
we apply the Sobolev embedding theorem and Lemma \ref{lemma3} with $k=p_2+\mu$ to deduce
\begin{align}
\int_{\Omega}\frac{\vert v\vert^p\log\vert v\vert}p\,\rd x
& =\left(\int_{\Omega_1}+\int_{\Omega_2}\right)\frac{\vert v\vert^p\log\vert v\vert}p\,\rd x
\leq\frac{1}{\e\mu p_1}\int_{\Omega_2}|v|^{p_2}|v|^{\mu}\,\rd x\nonumber\\
& \leq\frac{1}{\e\mu p_1}\int_{\Omega}|v|^{p_2+\mu}\,\rd x
\leq\frac{B_{p_2+\mu}}{\e\mu p_1}\Vert\nabla v\Vert_2^{p_2+\mu}.\label{r3}
\end{align}
Here we used the inequality $x^{-\mu}\log x<(\e\mu)^{-1}$ for $x\geq1$, where $\mu>0$ is the constant stated in the theorem. Then according to the definition \eqref{2.4} of $E(t)$, we substitute $v=u(t)$ in \eqref{r3} to estimate
\begin{align*}
E(t) &\geq\frac{\ell}{2}\Vert\nabla u(t)\Vert_2^2+\frac{1}{2}(g\circ\nabla u)(t)-\frac{\alpha B_{p_2+\mu}}{\e\mu p_1}\Vert\nabla u(t)\Vert_2^{p_2+\mu}+\frac{\alpha}{p_2^2}\int_{\Omega}|u(t)|^p\,\rd x\\
&\geq\frac{1}{2}\left\{\ell\Vert\nabla u(t)\Vert_2^2+(g\circ\nabla u)(t)\right\}-\frac{\alpha B}{\e\mu p_1}\left\{\ell\Vert\nabla u(t)\Vert_2^2+(g\circ\nabla u)(t)\right\}^{\frac{p_2+\mu}{2}}\\
& =\frac{1}{2}\lambda(t)^2-\frac{\alpha B}{\e\mu p_1}\lambda(t)^{p_2+\mu}=R(\lambda(t)),
\end{align*}
by the definitions of $B$, $\la(t)$ and $R(\la)$.
\end{proof}

It is easily verified that $R(\lambda)$ attains its maximum at
$$
\lambda_1:=\left(\frac{\e\mu p_1}{\alpha(p_2+\mu)B}\right)^{\frac{1}{p_2+\mu-2}}>0
$$
with the corresponding maximum
\[
E_1:=R(\lambda_1)=\left(\frac{1}{2}-\frac{1}{p_2+\mu}\right)\lambda_1^2>0.
\]

\begin{lemma}[see {\cite[Lemma 3.4]{23}}]\label{lemma 2.7}
Let assumption {\rm(A1)} hold and $p$ satisfies {\rm\eqref{1.2}--\eqref{1.3}}. Let $u$ be a solution to problem \eqref{1.1} whose initial data satisfies
\[
E(0)<E_1,\quad\lambda(0)<\lambda_1.
\]
Then there exists a constant $\lambda_2\in(0,\lambda_1)$ such that
\[
\lambda(t)=\left\{\ell\Vert\nabla u(t)\Vert_2^2+(g\circ\nabla u)(t)\right\}^{\frac{1}{2}}\leq\lambda_2,\quad\forall\,t\in[0,T_{\max}),
\]
where $T_{\max}$ is the maximal existence time.
\end{lemma}

Next, we invoke the existence result for the solution to \eqref{1.1}.

\begin{proposition}[Local existence]\label{theorem1}
Let $(u_0,u_1)\in H_0^1(\Omega)\times L^2(\Omega) $ be given. Let assumption {\rm(A1)} hold and $p$ satisfy {\rm\eqref{1.2}--\eqref{1.3}}. Then there exists $T>0$ such that problem \eqref{1.1} admits a unique local weak solution $u$ on $[0,T]$.
\end{proposition}

Combining the Faedo-Galerkin method with the proof for logarithmic nonlinearity in \cite{4} and that for nonlinear damping and a memory term in \cite{2}, we can show the local existence stated in Proposition \ref{theorem1}.

We introduce an auxiliary energy
\begin{equation}\label{r4}
\mathbb{E}(t):=E(t)+\alpha\int_{\Omega}\frac{|u(t)|^p\log|u(t)|}p\,\rd x.
\end{equation}
We prepare additional estimates for $E(t)$ and $\mathbb E(t)$.

\begin{lemma}\label{lemma2.8}
Under the same assumptions of Lemma $\ref{lemma 2.7},$ there hold
\begin{gather}
\alpha\int_{\Omega}\frac{|u(t)|^p\log|u(t)|}p\,\rd x\leq\wt{C}E(t)\leq\wt{C}E(0),\label{r1}\\
\mathbb{E}(t)\leq(1+\wt{C})E(t)\leq(1+\wt{C})E(0)\label{r2}
\end{gather}
for any $t\in [0,T_{\max}),$ where
$$
\wt{C}:=\frac{\frac{2\alpha\lambda_2^{p_2+\mu-2}B}{\e\mu p_1}}{1-\frac{2\alpha\lambda_2^{p_2+\mu-2}B}{\e\mu p_1}}.
$$
\end{lemma}

\begin{proof}
Since the second halves of \eqref{r1} and \eqref{r2} follow immediately from the monotonicity of $E(t)$, it suffices to verify the respective first halves.

We employ \eqref{r3} and Lemma \ref{lemma 2.7} to deduce
\begin{align*}
\alpha\int_{\Omega}\frac{|u(t)|^p\log|u(t)|}p\,\rd x &
\leq\frac{\alpha B_{p_2+\mu}}{\e\mu p_1}\Vert\nabla u(t)\Vert_2^{p_2+\mu}
\leq\frac{\alpha B}{\e\mu p_1}\left(\ell^{\frac{1}{2}}\Vert\nabla u(t)\Vert_2\right)^{p_2+\mu-2}\ell\Vert\nabla u(t)\Vert_2^2\\
& \leq\frac{2\alpha\lambda_2^{p_2+\mu-2}B}{\e\mu p_1}\left(E(t)+\alpha\int_{\Omega}\frac{|u(t)|^p\log|u(t)|}p\,\rd x\right),
\end{align*}
which implies the first half of \eqref{r1} by a simple rearrangement. Then the first half of \eqref{r2} is a direct consequence of \eqref{r1} and \eqref{r4}.
\end{proof}

\begin{remark}
Clearly, \eqref{r2} indicates that $\mathbb{E}(t)$ is uniformly bounded for all $t\in [0,T_{\max})$, implying the global existence of the solution, i.e., $T_{\max}=\infty$. At the same time, we also have $0\leq E(t)\leq E(0)$ for all $t\in[0,\infty)$.
\end{remark}


\section{Energy decay results}\label{sec3}

This section is devoted to the statements of our main results regarding the energy decay rates of global solutions to problem \eqref{1.1}. The first result is as follows.

\begin{theorem}\label{theorem3.2}
Under the same assumptions of Lemma $\ref{lemma 2.7},$ further let $\alpha$ be sufficiently small. then there exists a constant $C>0$ such that the energy $E(t)$ of the solution to \eqref{1.1} satisfies the following polynomial decay estimates
\begin{equation}\label{r3.1}
\int_0^{\infty}E(t)\,\rd t\leq C E(0),\quad
E(t)\leq C E(0)(1+t)^{-1},\quad t\geq0.
\end{equation}
\end{theorem}

Next, suppose that $g(t)$ satisfies the following additional assumption.\medskip

{\bf(A2)} Let $G\in C^1([0,\infty);[0,\infty))$ be either linear or strictly increasing, strictly convex and of $C^2$-class on $[0,r]$ ($r\le g(0)$) satisfying $G(0)=G'(0)=0$, and $\zeta\in C^1([0,\infty);(0,\infty))$ be non-increasing. The function $g(t)$ satisfies the ordinary differential inequality
$$
g'(t)\leq-\zeta(t)G(g(t)),\quad\forall\,t>0.
$$

\begin{remark}[see \cite{11}]
If $G$ is a strictly increasing and strictly convex $C^2$ function on $[0,r]$ satisfying $G(0)=G'(0)=0$, then it admits an extension $\ov{G}$ sharing the same monotonicity, convexity and regularity on $[0,\infty)$. For instance, if $G(r)=a$, $G'(r)=b$ and $G''(r)=c$, then we can construct $\ov{G}$ for $t>r$ as 
\[
\ov{G}(t)=\frac{C}{2}t^2+(b-C r)t+\left(a+\frac{C}{2}r^2-b r\right).
\]
Meanwhile, in this case the convexity of $G$ and $G(0)=0$ yield
\[
G(\theta t)\leq\theta G(t),\quad\forall\,\theta\in[0,1],\ \forall\,t\in(0,r].
\]
\end{remark}

With assumption (A2), we have the following result.

\begin{theorem}\label{the3.2}
Under the same assumptions of Theorem $\ref{theorem3.2},$ further let assumption {\rm(A2)} hold. Then there exist constants $k_1,k_2,t_1>0$ such that the energy of problem \eqref{1.1} satisfies
\begin{equation}\label{rr3.44}
E(t)\leq k_2G_1^{-1}\left(k_1\int_{t_1}^t\zeta(s)\,\rd s\right)
\end{equation}
for all $t\geq t_1,$ where $G_1(t):=\int_t^r\frac{1}{s\,G'(s)}\,\rd s$ is strictly decreasing and convex on $[0,r]$ and satisfies $\lim_{t\to0}G_1(t)=\infty$.
\end{theorem}

For better understanding, we provide an example to illustrate the result above.

\begin{example}
Let $g(t)=a\exp(-t^p)$, where $0<p<1$ and $a>0$ is sufficiently small so that $g$ satisfies assumption (A1). Then assumption (A2) is fulfilled with
$$
\zeta(t)=1,\quad G(t)=\frac{p^t}{(\ln\f a t)^{\f1p-1}}.
$$
Then direct calculation yields
$$
G'(t)=\frac{(1-p)+p\ln\f a t}{(\ln\f a t)^{\f1p}},\quad G''(t)=\frac{(1-p)(\ln\f a t+\f1p)}{(\ln\f a t)^{\frac{1}{p+1}}}.
$$
Thus, $G$ satisfies assumption (A2) on $[0,r]$ for any $0<r<a$. Further, we calculate
\begin{align*}
G_1(t) & =\int_t^r\frac{1}{s\,G'(s)}\,\rd s
=\int_t^r\frac{(\ln\frac{a}{s})^{\frac{1}{p}}}{s(1-p+p\ln\frac{a}{s})}\,\rd s
=\int_{\ln\frac{a}{r}}^{\ln\frac{a}{t}}\frac{\si^{\frac{1}{p}}}{1-p+p\si}\,\rd\si\\
& =\frac{1}{p}\int_{\ln\frac{a}{r}}^{\ln\frac{a}{t}}{\si^{\frac{1}{p}-1}}\left(\frac{\si}{\frac{1-p}{p}+\si}\right)\rd\si
\leq\frac{1}{p}\int_{\ln\frac{a}{r}}^{\ln\frac{a}{t}}{\si^{\frac{1}{p}-1}}\,\rd\si
\leq\left(\ln\frac{a}{t}\right)^{\frac{1}{p}}.
\end{align*}
Consequently, \eqref{rr3.44} implies $E(t)\leq k\exp(-k t^p)$.
\end{example}

Finally, we add an alternative assumption on $g(t)$ as follows.\medskip

{\bf(A3)} Let $\xi\in C^1([0,\infty);(0,\infty))$ be non-increasing such that $\int_0^\infty\xi(s)\,\rd s=\infty$, and $1\leq q<2$. The function $g(t)$ satisfies the ordinary differential inequality
\[
g'(t)\leq-\xi(t)g(t)^q,\quad\forall\,t>0.
\]

Replacing assumption (A2) with (A3), we obtain the following result.

\begin{theorem}\label{theorem3.1}
Under the same assumptions of Theorem $\ref{theorem3.2},$ further let assumption {\rm(A3)} hold. Then there exist constants $K>0$ and $K'>0$, such that the energy of problem \eqref{1.1} satisfies
\begin{equation}\label{3.2}
E(t)\leq\left\{\begin{alignedat}{2}
& E(0)\exp\left(1-K\int_0^t\xi(s)\,\rd s\right), &\quad & q=1,\\
& E(0)\left(\f q{1+K'(q-1)\int_0^t\xi(s)\,\rd s}\right)^{\frac{1}{q-1}}, & \quad & 1<q<2.
\end{alignedat}\right.
\end{equation}
\end{theorem}

Theorem \ref{theorem3.1} generalizes the results in \cite{25,24,qing}, where $g'(t)$ was assumed to satisfy condition (A3) with $\le p<\frac{3}{2}$, while our result holds for $1\leq q<2$. Moreover, the decay rates
\[
E(t)\leq\left\{\begin{alignedat}{2}
& K\exp\left(1-\lambda\int_{t_0}^t\xi(s)\,\rd s\right), & \quad & p=1,\\
& K\left(\frac{1}{1+\int_{t_0}^t\xi^{2p-1}(s)\,\rd s}\right)^{\frac{1}{2p-2}}, & \quad & 1<p<\frac{3}{2}
\end{alignedat}\right.
\]
obtained in \cite{24} are improved under this broader framework.


\section{Proof of energy decay results}\label{secc3}

Throughout this section, by $C>0$ we denote generic constants with may change line to line. 


\subsection{Proof of Theorem \ref{theorem3.2}}

To proceed, we shall first introduce several auxiliary functions and prepare several lemmas accordingly.

Recall the functions $I(t)$ and $M(\de)$ defined in \eqref{eq-IM} and introduce
$$
I_1(t):=\int_{\Omega}\int_0^t I(t-s)\vert\nabla u(s)\vert^2\,\rd s\rd x,\quad I_2(t):=M(\delta)(\delta I_1(t)+E(t)),\quad\de>0.
$$

\begin{lemma}[see {\cite[Lemma 2.1]{22}}]\label{le3.1}
For $t\geq0$, the functions $I_1(t),I_2(t)$ defined above satisfy
\begin{equation}\label{r3.3}
\begin{aligned}
\frac{\rd}{\rd t}I_1(t) & \leq-\frac{1}{2}(g\circ\nb u)(t)+2\|\nabla u(t)\|_2^2,\\
\frac{\rd}{\rd t}I_2(t) & \leq-\frac{M(\delta)}{2}\int_{\Omega}\int_0^t K_{\delta}(t-s)g(t-s)\vert\nabla u(t)-\nabla u(s)\vert^2\,\rd s\rd x+2\delta M(\delta)\|\nabla u(t)\|_2^2.
\end{aligned}
\end{equation}
Moreover, we have
\begin{equation}\label{rr3.5}
\delta M(\delta)\longrightarrow 0\quad\mbox{as }\delta\to0.
\end{equation}
\end{lemma}

Next, we further introduce
$$
I_3(t):=\int_{\Omega}u_t(t)u(t)\,\rd x,\quad I_4(t):=-\int_{\Omega}u_t(t)\int_0^t g(t-s)(u(t)-u(s))\,\rd s\rd x.
$$
The following two lemmas give estimates for the derivatives of $I_3(t)$ and $I_4(t)$.

\begin{lemma}\label{le3.2}
Let $(u_0,u_1)\in H_0^1(\Omega)\times L^2(\Omega)$ be given. Under the same assumptions of Lemma $\ref{lemma 2.7},$ the function $I_3(t)$ defined above satisfies
\begin{align}
\frac{\rd}{\rd t}I_3(t) & \leq-\frac{\ell}{2}\Vert\nabla u(t)\Vert_2^2+\left(1+\f{B_2}\ell\right)\Vert u_t(t)\Vert_2^2+\frac{1}{\ell}\int_{\Omega}\left\vert\int_0^t g(t-s)(\nabla u(s)-\nabla u(t))\,\rd s\right\vert^2\rd x\nonumber\\
& \quad\,+\alpha\int_{\Omega}|u(t)|^p\log|u(t)|\,\rd x.\label{rr3.3}
\end{align}
\end{lemma}

\begin{proof}
By direct differentiation and the governing equation in \eqref{1.1}, we obtain
\begin{align}
\frac{\rd}{\rd t}I_3(t) & \leq-\ell\Vert\nabla u(t)\Vert_2^2+\Vert u_t(t)\Vert_2^2+\int_{\Omega}\nabla u(t)\cdot\int_0^t g(t-s)(\nabla u(s)-\nabla u(t))\,\rd s\rd x\nonumber\\
& \quad\,-\int_{\Omega}u_t(t)u(t)\,\rd x+\alpha\int_{\Omega}|u(t)|^p\log|u(t)|\,\rd x.\label{r3.5}
\end{align}
For the third and fourth terms on the right-hand side of \eqref{r3.5}, we apply Cauchy's inequality to estimate
\begin{align*}
& \quad\,\int_{\Omega}\nabla u(t)\cdot\int_0^t g(t-s)(\nabla u(s)-\nabla u(t))\,\rd s\rd x\\
& \leq\frac{\ell}{4}\Vert\nabla u(t)\Vert_2^2+\frac{1}{\ell}\int_{\Omega}\left\vert\int_0^t g(t-s)(\nabla u(s)-\nabla u(t))\,\rd s\right\vert^2\rd x,
\end{align*}
and
\begin{align*}
-\int_{\Omega}u_t(t)u(t)\,\rd x & \leq\int_\Om|u_t(t)u(t)|\,\rd x
\le\int_{\Omega}\left(\frac{1}{4\delta}\vert u_t(t)\vert^2+\delta|u(t)|^2\right)\,\rd x\\
& \leq\frac{1}{4\delta}\|u_t(t)\|_2^{2}+\delta B_2\Vert\nabla u(t)\Vert_2^2,
\end{align*}
where we used Lemma \ref{lemma3} with $k=2$ and $\de>0$ is a constant. Substituting the above two inequalities into \eqref{r3.5}, we arrive at the desired estimate \eqref{rr3.3} by simply choosing $\delta=\frac{\ell}{4 B_2}$.
\end{proof}

\begin{lemma}\label{le3.3}
Let $(u_0,u_1)\in H_0^1(\Omega)\times L^2(\Omega)$ be given. Under the same assumptions of Lemma $\ref{lemma 2.7},$ the function $I_4(t)$ defined above satisfies
\begin{align}
\frac{\rd}{\rd t}I_4(t) & \leq\delta(1+\alpha\xi_1)\Vert\nabla u(t)\Vert_2^2+c(\de)\int_{\Omega}\left\vert\int_0^t g(t-s)(\nabla u(t)-\nabla u(s))\,\rd s\right\vert^2\rd x\nonumber\\
& \quad\,+\left(\f1{4\de}+\de-\int_0^t g(s)\,\rd s\right)\Vert u_t(t)\Vert_2^2+\frac{g(0)B_2}{4\delta}(-g'\circ\nabla u)(t),\label{xx3.1}
\end{align}
where $\de>0$ is a constant to be determined later, $c(\de)>0$ denotes a generic constant depending on $\de$ such that $\lim_{\de\to0}c(\de)=\infty$ and
\begin{equation}\label{eq-xi}
\xi_1:=\f{B_2}{\e(p_1-2)}+\f{B_{2(p_2-1+\mu)}}{\e\mu}\left(\frac{2(1+\wt{C})E(0)}{\ell}\right)^{p_2-2+\mu}.
\end{equation}
\end{lemma}

\begin{proof}
Directly differentiating $I_4(t)$ and employing the governing equation of \eqref{1.1} yield
\begin{align*}
\frac{\rd}{\rd t}I_4(t) & =-\int_{\Omega}u_{tt}\int_0^t g(t-s)(u(t)-u(s))\,\rd s\rd x\\
& \quad\,-\int_{\Omega}u_t(t)\int_0^t g'(t-s)(u(t)- u(s))\,\rd s\rd x-\int_0^t g(s)\,\rd s\Vert u_t(t)\Vert_2^2\\
& =:\sum_{i=1}^6I_2^i(t),
\end{align*}
where
\begin{align*}
I_2^1(t) & :=\int_{\Omega}\left(1-\int_0^t g(s)\,\rd s\right)\nabla u(t)\cdot\int_0^t g(t-s)(\nabla u(t)-\nabla u(s))\,\rd s\rd x,\\
I_2^2(t) & :=\int_{\Omega}\left\vert\int_0^t g(t-s)(\nabla u(t)-\nabla u(s))\,\rd s\right\vert^2\rd x,\\
I_2^3(t) & :=\int_{\Omega} u_t(t)\left(\int_0^t g(t-s)(u(t)-u(s))\,\rd s\right)\rd x,\\
I_2^4(t) & :=-\alpha\int_{\Omega}|u(t)|^{p-2}u(t)\log|u(t)|\left(\int_0^t g(t-s)(u(t)-u(s))\,\rd s\right)\,\rd x,\\
I_2^5(t) & :=-\int_{\Omega}u_t(t)\int_0^t g'(t-s)(u(t)- u(s))\,\rd s\rd x,\\
I_2^6(t) & :=-\int_0^t g(s)\,\rd s\,\Vert u_t(t)\Vert_2^2 .
\end{align*}

Now we estimate the each of the 6 terms above. For $I_2^1(t)$, we apply Cauchy's inequality with $\de>0$ to deduce
\begin{equation}\label{r3.10}
I_2^1(t)\leq\delta\Vert\nabla u(t)\Vert_2^2+\frac{1}{4\delta}\int_{\Omega}\left\vert\int_0^t g(t-s)(\nabla u(t)-\nabla u(s))\,\rd s\right\vert^2\rd x.
\end{equation}
Similarly, we estimate $I_2^3(t)$ as
\begin{align}
I_2^3(t) & \leq\delta\int_{\Omega}\left(\int_0^{t}g(t-s)(u(t)-u(s))\,\rd s\right)^{2}\rd x+\f1{4\delta}\|u_t(t)\|_2^{2}\nonumber\\
 & \leq B_2\delta\int_{\Omega}\left|\int_0^{t}g(t-s)(\nb u(t)-\nb u(s))\,\rd s\right|^{2}\rd x+\f1{4\delta}\|u_t(t)\|_2^{2},\label{r3.12}
\end{align}
where we used the Poincar\'e inequality. In a same manner, we first estimate $I_2^5(t)$ as
\begin{equation}\label{eq-I25}
I_2^5(t)\leq\delta\Vert u_t(t)\Vert_2^2+\f1{4\de}\int_\Om\left\vert\int_0^t g'(t-s)|u(t)-u(s)|\,\rd s\right\vert^2\rd x.
\end{equation}
By assumption (A1) on $g$, we apply the Cauchy-Schwarz inequality to treat
\begin{align*}
\left\vert\int_0^t g'(t-s)|u(t)-u(s)|\,\rd s\right\vert^2 & \le\int_0^t(-g'(t-s))\,\rd s\int_0^t(-g'(t-s))|u(t)-u(s)|^2\,\rd s\\
& \le g(0)\int_0^t(-g'(t-s))|u(t)-u(s)|^2\,\rd s.
\end{align*}
Then we employ the Poincar\'e inequality and recall the definition of $(g\circ\nb u)(t)$ to treat the second term on the right-hand side of \eqref{eq-I25} as
\[
\int_\Om\left\vert\int_0^t g'(t-s)|u(t)-u(s)|\,\rd s\right\vert^2\rd x\le g(0)\int_0^t(-g'(t-s))\|u(t)-u(s)\|_2^2\,\rd x\le g(0)B_2(-g'\circ\nb u)(t),
\]
indicating
\begin{equation}\label{r3.23}
I_2^5(t)\leq\delta\Vert u_t(t)\Vert_2^2+\frac{g(0)B_2}{4\delta}(-g'\circ\nabla u)(t).
\end{equation}

Next, we deal with $I_2^4(t)$. For $v\in H_0^1(\Om)$, we split $\Om$ as that in \eqref{eq-split} and utilize the assumption \eqref{1.2} on $p$ to obtain
\[
\left||v|^{p-2}v\log|v|\right|\le\begin{cases}
\left||v|^{p_1-2}v\log|v|\right| & \mbox{in }\Om_1,\\
\left||v|^{p_2-2}v\log|v|\right| & \mbox{in }\Om_2.
\end{cases}
\]
For $\Om_1$, we use the inequality $|x^{p_1-2}\log x|<\f1{\e(p_1-2)}$ for $0<x<1$ to bound
\[
\left||v|^{p-2}v\log|v|\right|\le\f{|v|}{\e(p_1-2)}\quad\mbox{in }\Om_1.
\]
For $\Om_2$, we take advantage of the inequality $|x^{-\mu}\log x|<\f1{\e\mu}$ for $x>1$, where $\mu>0$ is sufficiently small such that $2(p_2-1+\mu)<\f{2n}{n-2}$. Then we can bound
\[
\left||v|^{p-2}v\log|v|\right|\le\f{|v|^{p_2-1+\mu}}{\e\mu}\quad\mbox{in }\Om_2.
\]
Combining the above two cases, we obtain
\[
\left||v|^{p-2}v\log|v|\right|\leq\f{|v|}{\e(p_1-2)}+\f{|v|^{p_2-1+\mu}}{\e\mu}\quad\mbox{in }\Om.
\]
Therefore, we apply the above estimate with $v=u(t)$ to $I_2^4(t)$ and employ again Cauchy's inequality with $\de>0$ to deduce
\begin{align*}
I_2^4(t) & \leq\int_{\Omega}\left(\f{|u(t)|}{\e(p_1-2)}+\f{|u(t)|^{p_2-1+\mu}}{\e\mu}\right)\left|\int_0^{t}g(t-s)(u(t)-u(s))\,\rd s\right|\rd x\\
& \leq\delta\int_\Om\left(\f{|u(t)|^2}{\e(p_1-2)}+\f{|u(t)|^{2(p_2-1+\mu)}}{\e\mu}\right)\rd x\\
& \quad\,+\f1{4\de}\left(\f1{\e(p_1-2)}+\f1{\e\mu}\right)\int_{\Omega}\left\vert\int_0^{t}g(t-s)(u(t)-u(s))\,\rd s\right\vert^2\rd x\\
& \leq\delta\left(\f{\|u(t)\|_2^2}{\e(p_1-2)}+\f{\|u(t)\|_{2(p_2-1+\mu)}^{2(p_2-1+\mu)}}{\e\mu}\right)+\frac{C}{4\delta}\int_{\Omega}\left\vert\int_0^{t}g(t-s)(u(t)-u(s))\,\rd s\right\vert^2\rd x.
\end{align*}
Finally, we apply Lemma \ref{lemma3} with $k=2$ and $k={2(p_2-1+\mu)}$ to dominate
\begin{align*}
I_2^4(t) & \leq\delta\left(\f{B_2}{\e(p_1-2)}\Vert\nabla u(t)\Vert_2^{2}+\f{B_{2(p_2-1+\mu)}}{\e\mu}\Vert\nabla u(t)\Vert_2^{2(p_2-1+\mu)}\right)\\
& \quad\,+\frac{C}{4\delta}\int_{\Omega}\left\vert\int_0^{t}g(t-s)(\nb u(t)-\nb u(s))\,\rd s\right\vert^2\rd x.
\end{align*}
For $\|\nb u(t)\|_2^{2(p_2-1+\mu)}$, we further employ the same argument as that for Lemma \ref{lemma2.8} to estimate
\[
\|\nb u(t)\|_2^{2(p_2-1+\mu)}=\|\nb u(t)\|_2^{2(p_2-2+\mu)}\|\nb u(t)\|_2^2\le\left(\frac{2(1+\wt{C})E(0)}{\ell}\right)^{p_2-2+\mu}\Vert\nabla u(t)\Vert_2^{2}.
\]
Then we conclude
\begin{equation}\label{r3.22}
I_2^4(t) \leq\delta\xi_1\Vert\nabla u(t)\Vert_2^{2}+\frac{C}{4\delta}\int_{\Omega}\left\vert\int_0^{t}g(t-s)(\nabla u(t)-\nabla u(s))\,\rd s\right\vert^2\rd x,
\end{equation}
where $\xi_1$ was defined in \eqref{eq-xi}.

Combining all the estimates \eqref{r3.10}, \eqref{r3.12}, \eqref{r3.23} and \eqref{r3.22}, we eventually arrived at \eqref{xx3.1}.
\end{proof}

Now, we are well prepared to proceed to the proof of Theorem \ref{theorem3.2}. We start with introducing an auxiliary function
$$
L(t):=N_1E(t)+I_3(t)+N_2I_4(t),
$$
where $N_1,N_2>0$ are constants to be determined later. By H\"older's and Cauchy's inequalities, it is easy to verify that $L(t)\sim E(t)$, i.e., $L(t)$ is equivalent to $E(t)$ in the sense of energy comparison.

Fix $t_1>0$ and set $g_1:=\int_0^{t_1}g(s)\,\rd s>0$. Combining \eqref{2.5}, \eqref{rr3.3} and \eqref{xx3.1} with $\delta=\frac{\ell}{4N_2(1+\alpha\xi_1)}$, we estimate $L'(t)$ for all $t\geq t_1$ that
\begin{align*}
L'(t) & \leq-\frac{\ell}{4}\Vert\nabla u(t)\Vert_2^2
-\left\{N_1-1-\f{B_2}\ell-N_2\left(\f1{4\de}+\de-g_1\right)\right\}\|u_t(t)\|_2^2\\
& \quad\,+\left(\frac{N_1}{2}-\frac{g(0)B_2N_2}{4\delta}\right)(g'\circ\nabla u)(t)
+\alpha\int_{\Omega}|u(t)|^p\log|u(t)|\,\rd x\\
& \quad\,+\left(\f1\ell+c(\delta)N_2\right)\int_{\Omega}\left\vert\int_0^t g(t-s)(\nabla u(t)-\nabla u(s))\,\rd s\right\vert^2\rd x.
\end{align*}
At this point, we can choose sufficiently large $N_1,N_2>0$ such that 
$$
N_3:=N_1-1-\f{B_2}\ell-N_2\left(\f1{4\de}+\de-g_1\right)>0,\quad\frac{N_1}{2}-\frac{g(0)B_2N_2}{4\delta}>0.
$$
Then the negativity of $(g'\circ\nb u)(t)$ indicates
\begin{align}
L'(t) & \leq-\frac{\ell}{4}\Vert\nabla u(t)\Vert_2^2-N_3\|u_t(t)\|_2^2+\alpha\int_{\Omega}|u(t)|^p\log|u(t)|\,\rd x\nonumber\\
& \quad\,+N_4\int_{\Omega}\left\vert\int_0^t g(t-s)(\nabla u(t)-\nabla u(s))\,\rd s\right\vert^2\rd x,\label{4.5}
\end{align}
where $N_4:=\f1\ell+c(\delta)N_2$. Recalling the definition of $M(\delta)$ in \eqref{eq-IM}, we further bound
\begin{align*}
& \quad\,\int_{\Omega}\left\vert\int_0^t g(t-s)(\nabla u(t)-\nabla u(s))\,\rd s\right\vert^2\rd x\\
& \leq\int_{\Omega}\int_0^t\frac{g(s)}{K_{\delta}(s)}\,\rd s\int_0^t K_{\delta}(t-s)g(t-s)\vert\nabla u(t)-\nabla u(s)\vert^2\,\rd s\rd x\\
& \leq M(\delta)\int_{\Omega}\int_0^t K_{\delta}(t-s)g(t-s)\vert\nabla u(t)-\nabla u(s)\vert^2\,\rd s\rd x,
\end{align*}
which, together with \eqref{4.5}, indicates
\begin{align}
L'(t) & \leq-\frac{\ell}{4}\Vert\nabla u(t)\Vert_2^2-N_3\|u_t(t)\|_2^2+\alpha\int_{\Omega}|u(t)|^p\log|u(t)|\,\rd x\nonumber\\
& \quad\,+N_4M(\delta)\int_{\Omega}\int_0^t K_{\delta}(t-s)g(t-s)\vert\nabla u(t)-\nabla u(s)\vert^2\,\rd s\rd x.\label{4.6}
\end{align}

Next, we further introduce
$$
J(t):=L(t)+\frac{\ell}{32}I_1(t)+2N_4I_2(t).
$$
Then it follows from \eqref{r3.3} and \eqref{4.6} that
\begin{align*}
J'(t) & \leq-\left(\frac{\ell}{4}-\frac{\ell}{16}-2N_4\delta M(\delta)\right)\Vert\nabla u(t)\Vert_2^2-N_3\Vert u_t(t)\Vert_2^2-\frac{\ell}{64}(g\circ\nb u)(t)\\
& \quad\,+\alpha\int_{\Omega}|u(t)|^p\log|u(t)|\rd x .
\end{align*}
Owing to the convergence \eqref{rr3.5}, we can choose sufficiently small $\delta>0$ such that $2N_4\delta M(\delta)=\frac{\ell}{16}$. Therefore, in view of the original energy $E(t)$, we obtain
\begin{align}
J'(t) & \leq-\frac\ell8\Vert\nabla u(t)\Vert_2^2-N_3\|u_t(t)\|_2^2-\frac{\ell}{64}(g\circ\nb u)(t)+\alpha\int_{\Omega}|u(t)|^p\log|u(t)|\,\rd x\nonumber\\
& \leq-\ve E(t)-\left(\frac\ell8-\frac{\ve}{2}\right)\Vert\nabla u(t)\Vert_2^2-\left(N_3-\frac{\ve}{2}\right)\Vert u_t(t)\Vert_2^2-\left(\frac{\ell}{64}-\frac{\ve}{2}\right)(g\circ\nabla u)(t)\nonumber\\
& \quad\,+\alpha\int_{\Omega}\left(1-\frac{\ve}p\right)|u(t)|^p\log|u(t)|\,\rd x+\frac{\alpha \ve}{p_1^2}\int_{\Omega}|u(t)|^p\,\rd x,\label{4.7}
\end{align}
where $\ve\in(0, p_1)$ is a constant. Employing the same domain splitting argument as that for \eqref{r3}, we estimate 
\begin{align}
\int_{\Omega}\left(1-\frac{\ve}p\right)|u(t)|^p\log|u(t)|\,\rd x &
=\left(\int_{\Omega_1}+\int_{\Omega_2}\right)\left(1-\frac{\ve}p\right)|u(t)|^p\log|u(t)\,\rd x\nonumber\\
& \leq\frac{1}{\e\mu}\int_{\Omega_2}\left(1-\frac{\ve}p\right)|u(t)|^{p_2}|u(t)|^{\mu}\,\rd x
\leq\frac{1}{\e\mu}\int_{\Omega}|u(t)|^{p_2+\mu}\,\rd x\nonumber\\
& \leq\frac{B_{p_2+\mu}}{\e\mu}\Vert\nabla u(t)\Vert_2^{p_2+\mu}
\leq\xi_2\Vert\nabla u(t)\Vert_2^2,\label{4.8}
\end{align}
where
$$
\xi_2:=\frac{B_{p_2+\mu}}{\e\mu}\left(\frac{2(1+\wt{C})E(0)}{\ell}\right)^{\frac{p_2+\mu-2}{2}}$$
and $\mu>0$ is a sufficiently small constant such that $p_2+\mu<2_*=\frac{2n}{n-2}$. 

On the other hand, according to Lemma \ref{lemma2}, there exists a constant $B_{p(x)}>0$ such that
\[
\|v\|_{p(x)}\le B_{p(x)}\|\nb v\|_2,\quad\forall\,v\in H_0^1(\Om).
\]
Hence, making use of \eqref{2.1}, we deal with the last term in \eqref{4.7} in a similar manner as above as
\begin{align}
\int_{\Omega}|u(t)|^p\,\rd x & \leq\max\left\{\|u(t)\|_{p(x)}^{p_1},\|u(t)\|_{p(x)}^{p_2}\right\}\le\max\left\{B_{p(x)}^{p_1}\vert\vert\nabla u(t)\Vert_2^{p_1},B_{p(x)}^{p_2}\vert\vert\nabla u(t)\Vert_2^{p_2}\right\}\nonumber\\
& \leq\xi_3\Vert\nabla u(t)\Vert_2^{2},\label{4.9}
\end{align}
where
$$
\xi_3:=\max\left\{B_{p(x)}^{p_1}\left(\frac{2(1+\wt{C})E(0)}{\ell}\right)^{\frac{p_1-2}{2}},B_{p(x)}^{p_2}\left(\frac{2(1+\wt{C})E(0)}{\ell}\right)^{\frac{p_2-2}{2}}\right\}.
$$
Plugging \eqref{4.8} and \eqref{4.9} into \eqref{4.7}, we arrive at
\begin{align*}
J'(t) & \leq-\ve E(t)-\left(\frac\ell8-\frac{\ve}{2}-\al\xi_2-\f{\al\ve\xi_3}{p_1^2}\right)\Vert\nabla u(t)\Vert_2^2-\left(N_3-\frac{\ve}{2}\right)\Vert u_t(t)\Vert_2^2-\left(\frac{\ell}{64}-\frac{\ve}{2}\right)(g\circ\nabla u)(t).
\end{align*}
Choosing $\ve$ and $\alpha$ sufficiently small such that  
$$
\frac\ell8-\frac{\ve}{2}-\al\xi_2-\f{\al\ve\xi_3}{p_1^2}>0,\quad N_3-\frac{\ve}{2}>0,\quad\frac{\ell}{64}-\frac{\ve}{2}>0,
$$
we eventually conclude
\[
J'(t)\leq-\ve E(t).
\]
Since $J(t)\geq0$ for all $t\geq0$ and $J(t_0)\leq C E(0)$, we integrate the above inequality over $[t_0, \tau)$ for any $\tau>t_0$ to obtain
$$
\int_{t_0}^{\tau}E(t)\,\rd t\leq C E(0).
$$
Passing $\tau\to\infty$ yields the first estimate in \eqref{r3.1}, which, together with $E'(t)\leq0$, implies the second estimate in \eqref{r3.1}. The proof of Theorem \ref{theorem3.2} is completed.


\subsection{Proof of Theorem \ref{the3.2}}

The proof of Theorem \ref{the3.2} relies on the following key lemma.  

\begin{lemma}[see {\cite[Equation (44)]{q2}}]\label{lem3.6}
Let assumptions {\rm(A1)--(A2)} be satisfied. Then for all $t>0,$ there holds
\[
(g\circ\nabla u)(t)\leq\frac{1}{c}\ov{G}^{-1}\left(\frac{c(-g'\circ\nb u)(t)}{\zeta(t)}\right),
\]
where $c\in (0,1)$ is a sufficiently small constant, $G$ and $\zeta$ are defined in {\rm(A2)}, and $\ov{G}$ is a strictly increasing and strictly convex extension of $G$.
\end{lemma}

Now we can proceed to the proof of Theorem \ref{the3.2}. By the Cauchy-Schwarz inequality and assumption (A1), first we obtain
\begin{align*}
\int_{\Omega}\left\vert\int_0^t g(t-s)(\nabla u(t)-\nabla u(s))\,\rd s\right\vert^2\rd x
& \leq\int_{\Omega}\int_0^t g(s)\,\rd s\int_0^t g(t-s)\vert\nabla u(t)-\nabla u(s)\vert^2\,\rd s\rd x\\
& \leq(1-\ell)(g\circ\nabla u)(t),
\end{align*}
which, together with \eqref{4.6}--\eqref{4.9}, indicates
\[
L'(t)\leq-\varepsilon E(t)+C(g\circ\nabla u)(t).
\]
According to Lemma \ref{lem3.6}, we deduce
\[
L'(t)\leq-\varepsilon E(t)+C\ov{G}^{-1}\left(\frac{c(-g'\circ\nb u)(t)}{\zeta(t)}\right).
\]
Next, defining
\[
F(t):=\ov{G}'\left(\varepsilon_1\frac{E(t)}{E(0)}\right)L(t)
\]
with a constant $\varepsilon_1\in(0,r)$, we immediately see $F(t)\sim E(t)$. Noting that $\ov{G}''(t)\geq0$ and $E'(t)\leq0$, we obtain
\begin{align}
F'(t) & =\ov{G}'\left(\varepsilon_1\frac{E(t)}{E(0)}\right)L'(t)+\varepsilon_1\frac{E'(t)}{E(0)}\ov{G}''\left(\varepsilon_1\frac{E(t)}{E(0)}\right)L(t)\nonumber\\
& \leq-\varepsilon E(t)\ov{G}'\left(\varepsilon_1\frac{E(t)}{E(0)}\right)+C\ov{G}'\left(\varepsilon_1\frac{E(t)}{E(0)}\right)\ov{G}^{-1}\left(\frac{c(-g'\circ\nb u)(t)}{\zeta(t)}\right).\label{6.5}
\end{align}
Let $\ov{G}^*$ be the convex conjugate of $\ov{G}$ in the sense of Young, that is,
\begin{equation}\label{6.6}
\ov{G}^*(s):=s(\ov{G}')^{-1}(s)-(\ov{G}\circ(\ov{G}')^{-1})(s),\quad s\in(0,\ov{G}'(r)].
\end{equation}
Then it is known that $\ov{G}^*$ satisfies the generalized Young inequality 
\[
AB\leq\ov{G}^*(A)+\ov{G}(B),\quad\forall\,A\in(0,\ov{G}'(r)],\ \forall\,B\in (0,r].
\]
Thus, taking
\[
A=\ov{G}'\left(\varepsilon_1\frac{E(t)}{E(0)}\right),\quad B=\ov{G}^{-1}\left(\frac{c(-g'\circ\nb u)(t)}{\zeta(t)}\right)
\]
in the above inequality and using \eqref{6.6}, we can further estimate \eqref{6.5} as
\begin{align*}
F'(t) & \leq-\varepsilon E(t)\ov{G}'\left(\varepsilon_1\frac{E(t)}{E(0)}\right)+C\ov{G}^*\left(\ov{G}'\left(\varepsilon_1\frac{E(t)}{E(0)}\right)\right)+C\frac{(-g'\circ\nb u)(t)}{\zeta(t)}\\
& \leq-\varepsilon E(t)\ov{G}'\left(\varepsilon_1\frac{E(t)}{E(0)}\right)+C\varepsilon_1\frac{E(t)}{E(0)}\ov{G}'\left(\varepsilon_1\frac{E(t)}{E(0)}\right)+C\frac{(-g'\circ\nb u)(t)}{\zeta(t)}.
\end{align*}
Multiplying both sides of the above inequality by $\zeta(t)$ and using the facts that
$$
(-g'\circ\nb u)(t)\le-C E'(t),\quad\varepsilon_1\frac{E(t)}{E(0)}<r,\quad\ov{G}'\left(\varepsilon_1\frac{E(t)}{E(0)}\right)=G'\left(\varepsilon_1\frac{E(t)}{E(0)}\right),
$$
we derive
\begin{align}
\zeta(t)F'(t) & \leq-\varepsilon \zeta(t)E(t)\ov{G}'\left(\varepsilon_1\frac{E(t)}{E(0)}\right)+C\varepsilon_1\zeta(t)\frac{E(t)}{E(0)}\ov{G}'\left(\varepsilon_1\frac{E(t)}{E(0)}\right)+C(-g'\circ\nb u)(t)\nonumber\\
& \leq-(\varepsilon E(0)-C\varepsilon_1)\zeta(t)\frac{E(t)}{E(0)}G'\left(\varepsilon_1\frac{E(t)}{E(0)}\right)-C E'(t).\label{6.9}
\end{align}

Next, we further introduce $F_1(t):=\zeta(t)F(t)+C E(t)$. Then it is readily seen that $F_1(t)\sim E(t)$, i.e., there exist constants $\beta_1,\beta_2>0$ such that
\begin{equation}\label{6.10}
\beta_1F_1(t)\leq E(t)\leq\beta_2 F_1(t).
\end{equation}
Choosing $\varepsilon_1>0$ sufficiently small such that $\beta_3:=\varepsilon E(0)-C\varepsilon_1>0$, we differentiate $F_1(t)$ and utilize \eqref{6.9} to deduce
\begin{equation}
F_1'(t)\leq-\beta_3\zeta(t)\frac{E(t)}{E(0)}G'\left(\varepsilon_1\frac{E(t)}{E(0)}\right)
=-\beta_3{\zeta(t)}G_2\left(\frac{E(t)}{E(0)}\right),
\label{6.11}
\end{equation}
where we put $G_2(t):=t\,G'(\varepsilon_1t)$. Since $G_2'(t)=G'(\varepsilon_1t)+\varepsilon_1t\,G''(\varepsilon_1t)$, it follows from the strict monotonicity and the strict convexity of $G$ on $(0,r]$ that $G_2(t)>0$ and $G_2'(t)>0$ for all $t\in(0,1]$. Therefore, further defining $F_2(t):=\frac{\beta_1F_1(t)}{E(0)}$, we see $F_2(t)\sim E(t)$ and the combination of \eqref{6.10} and \eqref{6.11} implies
\[
F_2'(t)\leq-k_1\zeta(t)G_2\left(\frac{E(t)}{E(0)}\right)\leq-k_1\zeta(t)G_2\left(\frac{\beta_1F_1(t)}{E(0)}\right)=-k_1\zeta(t)G_2(F_2(t)),
\]
where $k_1:=\frac{\beta_1\beta_3}{E(0)}$. Hence, integrating both sides over $(t_1,t]$ yields
\[
\int_{\varepsilon_1F_2(t)}^{\varepsilon_1F_2(t_1)}\frac{1}{\tau G'(\tau)}\,\rd\tau=\int_{t_1}^t-\frac{F'_2(s)}{G_2(F_2(s))}\,\rd s\geq k_1\int_{t_1}^t\zeta(s)\,\rd s,
\]
where we performed a change of variables with $\tau=F_2(s)$. Finally, we conclude
\[
F_2(t)\leq\frac{1}{\varepsilon_1}G_1^{-1}\left(k_1\int_{t_1}^t\zeta(s)\,\rd s\right),
\]
where $G_1(t):=\int_t^r\frac{1}{s\,G'(s)}\,\rd s$ and we used the fact that $G_1$ is strictly decreasing on $(0,r]$. This eventually lead us to \eqref{rr3.44} in view of the equivalence $F_2(t)\sim E(t)$, which finalized the proof of Theorem \ref{the3.2}.


\subsection{Proof of Theorem \ref{theorem3.1}}

Similarly as the previous subsection, we first prepare a key lemma for the proof of Theorem \ref{theorem3.1}.

\begin{lemma}[see \cite{16}]\label{lemma3.1}
Let $E:[0,\infty)\longrightarrow(0,\infty)$ be a non-increasing function and $\phi:[0,\infty)\longrightarrow[0,\infty)$ be a strictly increasing $C^1$ function such that 
$$
\phi(0)=0,\quad\lim_{t\to\infty}\phi(t)=\infty.
$$
Suppose that there exist constants $\sigma\geq0$ and $\omega>0$ such that  
$$
\int_t^\infty\phi'(s)E(s)^{1+\sigma}\,\rd s\le\f1\om E(0)^\si E(t).
$$
Then the following decay estimates hold for $E(t)$.

$(1)$ If $\sigma=0,$ then $E(t)\leq E(0)\,\e^{1-\omega\phi(t)}$ for all $t\geq0$.

$(2)$ If $\sigma>0,$ then $E(t)\leq E(0)(\frac{1+\sigma}{1+\omega\sigma\phi(t)})^{\frac{1}{\sigma}}$ for all $t\geq0$.
\end{lemma}

We now proceed to the proof of Theorem \ref{theorem3.1} based on the above lemma.

Fixing some $T>0$, we multiply the governing equation in \eqref{1.1} by $\xi(t)E(t)^{q-1}u(t)$ and integrate over $\Omega\times(\tau,T)$ with some $\tau\in(0,T)$ to derive
\begin{align}
& \quad\,\int_\tau^T\xi(t)E(t)^{q-1}\int_{\Omega}u_{tt}(t)u(t)\,\rd x\rd t+\int_\tau^T\xi(t)E(t)^{q-1}\int_{\Omega}u_t(t)u(t)\,\rd x\rd t\nonumber\\
& \quad\,+\int_\tau^T\xi(t)E(t)^{q-1}\Vert\nabla u(t)\Vert_2^2\,\rd t-\int_\tau^T\xi(t)E(t)^{q-1}\int_{\Omega}\nabla u(t)\cdot\int_0^t g(t-s)\nabla u(s)\,\rd s\rd x\rd t\nonumber\\
& =\alpha\int_\tau^T\xi(t)E(t)^{q-1}\int_{\Omega}|u(t)|^p\log|u(t)|\,\rd x\rd t,\label{3.4}
\end{align}
where it is straightforward to see
\begin{align*}
& \quad\,\int_\tau^T\xi(t)E(t)^{q-1}\int_{\Omega}\nabla u(t)\cdot\int_0^t g(t-s)\nabla u(s)\,\rd s\rd x\rd t\\
& =\int_\tau^T\xi(t)E(t)^{q-1}\int_{\Omega}\int_0^t g(t-s)(\nabla u(s)-\nabla u(t))\cdot\nabla u(t)\,\rd s\rd x\rd t\\
& \quad\,+\int_\tau^T\xi(t)E(t)^{q-1}\int_0^t g(s)\,\rd s\Vert\nabla u(t)\Vert_2^2\,\rd t.
\end{align*}
Then we substitute the above identity into \eqref{3.4} and rearrange to obtain
\begin{align*}
& \quad\,\int_\tau^T\xi(t)E(t)^{q-1}\left(1-\int_0^t g(s)\,\rd s\right)\Vert\nabla u(t)\Vert_2^2\,\rd t\\
& =-\int_\tau^T\xi(t)E(t)^{q-1}\frac{\rd}{\rd t}\int_{\Omega}u_t(t)u(t)\,\rd x\rd t+\int_\tau^T\xi(t)E(t)^{q-1}\Vert u_t(t)\Vert_2^2\,\rd t\\
& \quad\,+\int_\tau^T\xi(t)E(t)^{q-1}\int_{\Omega}\int_0^t g(t-s)(\nabla u(s)-\nabla u(t))\cdot\nabla u(t)\,\rd s\rd x\rd t\\
& \quad\,-\int_\tau^T\xi(t)E(t)^{q-1}\int_{\Omega}u_t(t)u(t)\,\rd x\rd t
+\alpha\int_\tau^T\xi(t)E(t)^{q-1}\int_{\Omega}|u(t)|^p\log|u(t)|\,\rd x\rd t.
\end{align*}
Then we turn to the definition of $E(t)$ in \eqref{2.4} and employ the above identity to calculate
\begin{align}
2\int_\tau^T\xi(t)E(t)^q\,\rd t & =\int_\tau^T\xi(t)E(t)^{q-1}\left\{\|u_t(t)\|_2^2+\left(1-\int_0^t g(s)\,\rd s\right)\|\nb u(t)\|_2^2+(g\circ\nb u)(t)\right.\nonumber\\
& \qquad\qquad\qquad\qquad\;\;\;-\left.2\al\int_\Om\f{|u(t)|^p\log|u(t)|}p\,\rd x+2\al\int_\Om\f{|u(t)|^p}{p^2}\,\rd x\right\}\rd t\nonumber\\
& =\sum_{i=1}^7J_i(\tau),\label{3.7}
\end{align}
where
\begin{align*}
J_1(\tau) & :=2\int_\tau^T\xi(t)E(t)^{q-1}\|u_t(t)\|_2^2\,\rd t,\\
J_2(\tau) & :=-\int_\tau^T\xi(t)E(t)^{q-1}\int_{\Omega}u_t(t)u(t)\,\rd x\rd t,\\
J_3(\tau) & :=-\int_\tau^T\xi(t)E(t)^{q-1}\f\rd{\rd t}\int_\Om u_t(t)u(t)\,\rd x\rd t,\\
J_4(\tau) & :=\int_\tau^T\xi(t)E(t)^{q-1}(g\circ\nb u)(t)\,\rd t,\\
J_5(\tau) & :=\int_\tau^T\xi(t)E(t)^{q-1}\int_{\Omega}\int_0^t g(t-s)(\nabla u(s)-\nabla u(t))\cdot\nabla u(t)\,\rd s\rd x\rd t,\\
J_6(\tau) & :=\al\int_\tau^T\xi(t)E(t)^{q-1}\int_\Om\left(1-\f2p\right)|u(t)|^p\log|u(t)|\,\rd x\rd t,\\
J_7(\tau) & :=2\al\int_\tau^T\xi(t)E(t)^{q-1}\int_\Om\f{|u(t)|^p}{p^2}\,\rd x\rd t.
\end{align*}
In the sequel, we estimate each of the 7 terms above.\medskip


{\bf Step 1 } We first estimate $J_1(\tau)$ and  $J_2(\tau)$. Since \eqref{2.5} implies $\|u_t(t)\|_2^2\le-E'(t)$, it follows immediately from the monotonicity of $\xi(t)$ and $E(t)$ that
\begin{align}
J_1(\tau) & \leq-2\int_\tau^T\xi(t)E(t)^{q-1}E'(t)\,\rd t\leq2\xi(0)\left[\f{E(t)^q}q\right]_{t=T}^{t=\tau}\nonumber\\
& \le\f{2\xi(0)}q E(\tau)^q\le\f{2\xi(0)E(0)^{q-1}}q E(\tau).\label{3.12}
\end{align}
For $J_2(\tau)$, we apply Cauchy's inequality with $\de\in(0,1)$ and the above inequality to deduce
\begin{align*}
J_2(\tau) & \le\int_\tau^T\xi(t)E(t)^{q-1}\int_\Om\left(\f1{4\de}|u_t(t)|^2+\de|u(t)|^2\right)\rd x\rd t=\f1{8\de}J_1(\tau)+\de\int_\tau^T\xi(t)E(t)^{q-1}\|u(t)\|_2^2\,\rd t\\
& \le\f{\xi(0)E(0)^{q-1}}{4\de q}E(\tau)+\de B_2\int_\tau^T\xi(t)E(t)^{q-1}\|\nb u(t)\|_2^2\,\rd t,
\end{align*}
where again we used the Poincar\'e inequality. To deal with the last term above, we recall the definition \eqref{r4} of $\mathbb E(t)$ and apply Lemma \ref{lemma2.8} to bound
\begin{equation}\label{3.8}
\f\ell2\|\nb u(t)\|_2^2\le\f12\left(1-\int_0^t g(s)\,\rd s\right)\|\nb u(t)\|_2^2\le\mathbb E(t)\le(1+\wt C)E(t).
\end{equation}
Thus it reveals that
\begin{equation}\label{3.19}
J_2(t)\leq\frac{\xi(0)E(0)^{q-1}}{4\de q}E(\tau)+\frac{2\de B_2(1+\wt C)}\ell\int_\tau^T\xi(t)E(t)^q\,\rd t.
\end{equation}


{\bf Step 2 } Next, we deal with $J_3(\tau)$. Differentiating $\xi(t)E(t)^{q-1}\int_\Om u_t(t)u(t)\,\rd x$, we divide $J_3(\tau)=J_3^1(\tau)+J_3^2(\tau)+J_3^3(\tau)$ with
\begin{align*}
J_3^1(\tau) & :=-\int_\tau^T\f\rd{\rd t}\left(\xi(t)E(t)^{q-1}\int_\Om u_t(t)u(t)\,\rd x\right)\rd t,\\
J_3^2(\tau) & :=(q-1)\int_\tau^T\xi(t)E(t)^{q-2}E'(t)\int_\Om u_t(t)u(t)\,\rd x\rd t\quad(q>1),\\
J_3^3(\tau) & :=\int_\tau^T\xi'(t)E(t)^{q-1}\int_\Om u_t(t)u(t)\,\rd x\rd t.\\
\end{align*}
For $J_3^1(\tau)$, again we employ the monotonicity of $\xi(t),E(t)$ and the Poincar\'e inequality to deduce
\begin{align*}
J_3^1(\tau) & =\left[\xi(t)E(t)^{q-1}\int_\Om u_t(t)u(t)\,\rd x\right]_{t=T}^{t=\tau}\le\xi(0)E(0)^{q-1}\int_\Om(|u_t(\tau)u(\tau)|+|u_t(T)u(T)|)\,\rd x\\
& \le\f{\xi(0)E(0)^{q-1}}2\left(\|u_t(\tau)\|_2^2+\|u(\tau)\|_2^2+\|u_t(T)\|_2^2+\|u(T)\|_2^2\right)\rd x\\
& \le\f{\xi(0)E(0)^{q-1}}2\left\{\|u_t(\tau)\|_2^2+\|u_t(T)\|_2^2+B_2\left(\|\nb u(\tau)\|_2^2+\|\nb u(T)\|_2^2\right)\right\}.
\end{align*}
Similarly to the argument for \eqref{3.8}, we have
\[
\|u_t(t)\|_2^2\le2(1+\wt C)E(t),\quad\|\nb u(t)\|_2^2\le\f{2(1+\wt C)}\ell E(t),\quad\forall\,t\ge0,
\]
and hence
\begin{equation}\label{3.9}
J_3^1(\tau)\le2(1+\wt C)\left(1+\f{B_2}\ell\right)\xi(0)E(0)^{q-1} E(\tau).
\end{equation}
In an analogous manner, we estimate $J_3^2(t)$ as
\begin{align*}
J_3^2(\tau) & \le-(q-1)\int_\tau^T\xi(t)E(t)^{q-2}E'(t)\int_\Om|u_t(t)u(t)|\,\rd x\rd t\nonumber\\
& \le-\f{q-1}2\int_\tau^T\xi(t)E(t)^{q-2}E'(t)\left(\|u_t(t)\|_2^2+B_2\|\nb u(t)\|_2^2\right)\rd t\\
& \le-(q-1)(1+\wt C)\left(1+\f{B_2}\ell\right)\int_\tau^T\xi(t)E(t)^{q-1} E'(t)\,\rd t.
\end{align*}
Owing to the estimate of $\int_\tau^T\xi(t)E(t)^{q-1} E'(t)\,\rd t$ obtained in \eqref{3.12}, we immediately get
\begin{align}
J_3^2(\tau) & \le\f{(q-1)(1+\wt C)\xi(0)E(0)^{q-1}}q\left(1+\f{B_2}\ell\right)E(\tau)\nonumber\\
& \le(1+\wt C)\xi(0)E(0)^{q-1}\left(1+\f{B_2}\ell\right)E(\tau).\label{3.10}
\end{align}
For $J_3^3(\tau)$, we repeat the same argument and perform integration by parts to derive
\begin{align}
J_3^3(\tau) & \le\f12\int_\tau^T\xi'(t)E(t)^{q-1}\left(\|u_t(t)\|_2^2+B_2\|\nb u(t)\|_2^2\right)\rd t
\le-(1+\wt C)\left(1+\f{B_2}\ell\right)\int_\tau^T\xi'(t)E(t)^q\,\rd t\nonumber\\
& =(1+\wt C)\left(1+\f{B_2}\ell\right)\left\{\Big[\xi(t)E(t)^q\Big]_{t=T}^{t=\tau}+q\int_\tau^T\xi(t)E(t)^{q-1} E'(t)\,\rd t\right\}\nonumber\\
& \le(1+\wt C)\left(1+\f{B_2}\ell\right)\xi(0)E(0)^{q-1} E(\tau).\label{3.11}
\end{align}
Therefore, summing up \eqref{3.9}--\eqref{3.11} yields
\begin{equation}\label{eq-est-J3}
J_3(\tau)\le4(1+\wt C)\left(1+\f{B_2}\ell\right)\xi(0)E(0)^{q-1} E(\tau).
\end{equation}


{\bf Step 3 } To investigate $J_4(\tau)$ and $J_5(\tau)$, we shall divide the estimates into 2 cases, i.e., $q=1$ and $1<q<2$.\medskip

{\bf Case 1 } For $q=1$, assumption (A3) reduces to $g'(t)\le-\xi(t)g(t)$. Since \eqref{2.5} gives $(-g'\circ\nb u)(t)\le-2E'(t)$, the monotonicity of $\xi(t)$ indicates
\begin{align}
J_4(\tau) & \le\int_\tau^T\!\!\!\int_0^t\xi(t-s)g(t-s)\|\nb u(s)-\nb u(t)\|_2^2\,\rd s\rd t
\le\int_\tau^T(-g'\circ\nb u)(t)\,\rd t\nonumber\\
& \le-2\int_\tau^T E'(t)\,\rd t\le2E(\tau).\label{3.14}
\end{align}
As for $J_5(\tau)$, we apply Cauchy's inequality with $\ve>0$ and make use of \eqref{3.8} and \eqref{3.14} to derive
\begin{align}
J_5(\tau) & \le\f\ve2\int_\tau^T\xi(t)\int_0^t g(t-s)\|\nb u(t)\|_2^2\,\rd s\rd t
+\f1{2\ve}\int_\tau^T\!\!\!\int_0^t\xi(t-s)g(t-s)\|\nb u(s)-\nb u(t)\|_2^2\,\rd s\rd t\nonumber\\
& \le\f{\ve(1+\wt C)(1-\ell)}\ell\int_\tau^T\xi(t)E(t)\,\rd t+\f1\ve E(\tau)
\le\ve(1+\wt C)\int_\tau^T\xi(t)E(t)\,\rd t+\f1\ve E(\tau).\label{3.17}
\end{align}

{\bf Case 2 } Now let us consider the case of $1<q<2$. Applying H\"older's inequality with $q$ and $q^*=\f q{q-1}$, we employ \eqref{3.8} and \eqref{r3.1} to deduce
\begin{align*}
(g\circ\nb u)(t) & =\int_0^t\|\nb u(s)-\nb u(t)\|_2^{\f{2(q-1)}q}g(t-s)\|\nb u(s)-\nb u(t)\|_2^{\f2q}\,\rd s\\
& \le\left(\int_0^t\|\nb u(s)-\nb u(t)\|_2^2\,\rd s\right)^{1-\f1q}\left(\int_0^t g(t-s)^q\|\nb u(s)-\nb u(t)\|_2^2\,\rd s\right)^{\f1q}\\
& \le\left(2\int_0^t\left(\|\nb u(t)\|_2^2+\|\nb u(s)\|_2^2\right)\right)^{1-\f1q}(g^q\circ\nb u)(t)^{\f1q}\\
& \le C\left(\int_0^t(E(t)+E(s))\,\rd s\right)^{1-\f1q}(g^q\circ\nb u)(t)^{\f1q}
\le C E(0)^{1-\f1q}(g^q\circ\nb u)(t)^{\f1q}.
\end{align*}
Then we apply Young's inequality with $\ep>0$ to estimate $J_4(\tau)$ as
\begin{align*}
J_4(\tau) & \le C E(0)^{1-\f1q}\int_\tau^T\xi(t)E(t)^{q-1}(g^q\circ\nb u)(t)^{\f1q}\,\rd t\\
& =C E(0)^{1-\f1q}\int_\tau^T(\xi(t)E(t)^q)^{\f{q-1}q}\left(\xi(t)(g^q\circ\nb u)(t)\right)^{\f1q}\,\rd t\\
& \le C E(0)^{1-\f1q}\ep\int_\tau^T\xi(t)E(t)^q\,\rd t+C(\ep)\int_\tau^T\xi(t)(g^q\circ\nb u)(t)\,\rd t.
\end{align*}
For the last term above, we utilize assumption (A3) and the same argument for \eqref{3.14} to deduce
\begin{align*}
\int_\tau^T\xi(t)(g^q\circ\nb u)(t)\,\rd t &
\le\int_\tau^T\!\!\!\int_0^t\xi(t-s)g(t-s)^q\|\nb u(s)-\nb u(t)\|_2^2\,\rd s\rd t\\
& \le\int_\tau^T(-g'\circ\nb u)(t)\,\rd t\le-2\int_\tau^T E'(t)\,\rd t\le2E(\tau),
\end{align*}
and hence
\begin{equation}\label{3.16}
J_4(\tau)\le C\ep E(0)^{1-\f1q}\int_\tau^T\xi(t)E(t)^q\,\rd t+C(\ep)E(\tau).
\end{equation}
As for $J_5(\tau)$, we mimic the estimate in \eqref{3.17} and use \eqref{3.16} to derive
\begin{align}
J_5(\tau) & \le\f\ve2\int_\tau^T\xi(t)E(t)^{q-1}\int_0^t g(t-s)\|\nb u(t)\|_2^2\,\rd s\rd t+\f1{2\ve}\int_\tau^T\xi(t)E(t)^{q-1}(g\circ\nb u)(t)\,\rd t\nonumber\\
& \le\ve(1+\wt C)\int_\tau^T\xi(t)E(t)^q\,\rd t+\f1{2\ve}J_4(\tau)\nonumber\\
& \le\ve(1+\wt C)\int_\tau^T\xi(t)E(t)^q\,\rd t+\f1{2\ve}\left\{C\ep E(0)^{1-\f1q}\int_\tau^T\xi(t)E(t)^q\,\rd t+C(\ep)E(\tau)\right\}\nonumber\\
& \le\ve\left(1+\wt C+C E(0)^{1-\f1q}\right)\int_\tau^T\xi(t)E(t)^q+C(\ve)E(\tau),\label{3.18}
\end{align}
where we took $\ep=\ve^2$ in the last inequality.\medskip

{\bf Step 4 } Now it remains to deal with the two nonlinear terms $J_6(\tau)$ and $J_7(\tau)$. To start with, again we take advantage of the same argument in \eqref{eq-split}--\eqref{r3} to estimate
\begin{align*}
\int_\Om\left(1-\f2p\right)|u(t)|^p\log|u(t)|\,\rd x & \le\int_{\{|u(t)|\ge1\}}\left(1-\f2p\right)|u(t)|^p\log|u(t)|\,\rd x\\
& \le\left(1-\f2{p_2}\right)\int_{\{|u(t)|\ge1\}}|u(t)|^{p_2}\log|u(t)|\,\rd t\\
& \le\left(1-\f2{p_2}\right)\f1{\e\mu}\int_\Om|u(t)|^{p_2+\mu}\,\rd x
\le\f{(p_2-2)B_{p_2+\mu}}{\e\mu p_2}\|\nb u(t)\|_2^{p_2+\mu},
\end{align*}
where $\mu>0$ satisfies $p_2+\mu<2_*$. Then it follows from \eqref{3.8} that
\begin{align}
J_6(\tau) & \le\al\f{(p_2-2)B_{p_2+\mu}}{\e\mu p_2}\int_\tau^T\xi(t)E(t)^{q-1}\|\nb u(t)\|_2^{p_2+\mu}\,\rd x\nonumber\\
& \le\al\f{(p_2-2)B_{p_2+\mu}}{\e\mu p_2}\left(\f{2(1+\wt C)}\ell\right)^{\f{p_2+\mu}2}\int_\tau^T\xi(t)E(t)^{q+\f{p_2+\mu-2}2}\,\rd t
\le\al\xi_4\int_\tau^T\xi(t)E(t)^q\,\rd t,\label{3.21}
\end{align}
where
$$
\xi_4:=\f{(p_2-2)B_{p_2+\mu}}{\e\mu p_2}\left(\f{2(1+\wt C)}\ell\right)^{\f{p_2+\mu}2}E(0)^{\f{p_2+\mu-2}2}.
$$
Finally, for $J_7(\tau)$, we invoke the estimate \eqref{4.9} to deduce
\begin{align}
J_7(\tau) & \le\f{2\al}{p_1^2}\int_\tau^T\xi(t)E(t)^{q-1}\int_\Om|u(t)|^p\,\rd x\rd t
\le\f{2\al\xi_3}{p_1^2}\int_\tau^T\xi(t)E(t)^{q-1}\|\nb u(t)\|_2^2\,\rd t\nonumber\\
& \leq\frac{4\al(1+\wt C)\xi_3}{p_1^2\ell}\int_\tau^T\xi(t)E(t)^q\,\rd t.\label{3.22}
\end{align}

At this stage, we are in a position to complete the proof of Theorem \ref{theorem3.1}.\medskip

{\bf Case 1 } For $q=1$, we substitute \eqref{3.12}, \eqref{3.19}, \eqref{eq-est-J3}--\eqref{3.17}, \eqref{3.21} and \eqref{3.22} into \eqref{3.7} to dominate
\begin{align}
2\int_\tau^T\xi(t)E(t)\,\rd t
&\le\left\{2+\f1\ve+\left(2+4(1+\wt C)\left(1+\f{B_2}\ell\right)+\f1{4\de }\right)\xi(0) \right\}E(\tau)\nonumber\\
& \quad\,+\left\{(1+\wt C)\left(\ve+\f{2\de B_2}\ell\right)+\al\left(\xi_4+\f{4(1+\wt C)\xi_3}{p_1^2\ell}\right)\right\}\int_\tau^T\xi(t)E(t)\,\rd t.\label{3.23}
\end{align}
Therefore, if $\al>0$ is sufficiently small such that  
\begin{equation}\label{eq-alpha}
\al<\left(\xi_4+\f{4(1+\wt C)\xi_3}{p_1^2\ell}\right)^{-1},
\end{equation}
then we can suitably choose sufficiently small $\ve,\de>0$ such that 
\[
(1+\wt C)\left(\ve+\f{2\de B_2}\ell\right)+\al\left(\xi_4+\f{4(1+\wt C)\xi_3}{p_1^2\ell}\right)=1.
\]
Thus, we can rearrange \eqref{3.23} as
\[
\int_\tau^T\xi(t)E(t)\,\rd t\le\f1K E(\tau),\quad K:=\left\{2+\f1\ve+\left(2+4(1+\wt C)\left(1+\f{B_2}\ell\right)+\f1{4\de}\right)\xi(0)\right\}^{-1}.
\]
Since $T>0$ was chosen arbitrarily, we passing $T\to\infty$ to obtain
\[
\int_\tau^\infty\xi(t)E(t)\,\rd t\le\f1K E(\tau).
\]
Finally, setting $\phi(t)=\int_0^t\xi(s)\,\rd s$, $\si=0$ and $\om=K$ in Lemma \ref{lemma3.1}, we arrive at the first inequality in \eqref{3.2}.\medskip

{\bf Case 2 } For $1<q<2$, we follow the same line as above to substitute \eqref{3.12}, \eqref{3.19}, \eqref{eq-est-J3} and \eqref{3.16}--\eqref{3.22} into \eqref{3.7} to dominate
\begin{align}
2\int_\tau^T\xi(t)E(t)^q\,\rd t
& \le\left\{C(\ve)+\xi(0)E(0)^{q-1}\left(\f2q+\f1{4\de q}+4(1+\wt C)\left(1+\f{B_2}\ell\right)\right)\right\}E(\tau)\nonumber\\
& \quad\,+\left\{\f{2\de B_2(1+\wt C)}\ell+C\ve^2E(0)^{1-\f1q}+\ve\left(1+\wt C+C E(0)^{1-\f1q}\right)\right.\nonumber\\
& \qquad\;\;\,\,\left.+\al\left(\xi_4+\f{4(1+\wt C)\xi_3}{p_1^2\ell}\right)\right\}\int_\tau^T\xi(t)E(t)\,\rd t,\label{eq-case2}
\end{align}
where we set $\ep=\ve^2$ in \eqref{3.16}. Therefore, if $\al>0$ satisfies the same condition \eqref{eq-alpha} as above, we can still choose $\ve,\de>0$ suitably such that
\[
\f{2\de B_2(1+\wt C)}\ell+C\ve^2E(0)^{1-\f1q}+\ve\left(1+\wt C+C E(0)^{1-\f1q}\right)+\al\left(\xi_4+\f{4(1+\wt C)\xi_3}{p_1^2\ell}\right)=1.
\]
Then we can rearrange \eqref{eq-case2} as
\[
\int_\tau^T\xi(t)E(t)^q\,\rd t\le\f1{K'}E(0)^{q-1}E(\tau),
\]
where
\[
K':=\left\{C(\ve)+\xi(0)E(0)^{q-1}\left(\f2q+\f1{4\de q}+4(1+\wt C)\left(1+\f{B_2}\ell\right)\right)\right\}^{-1}.
\]
Thus, passing $T\to\infty$ yields
\[
\int_\tau^\infty\xi(t)E(t)^q\,\rd t\le\f1{K'}E(0)^{q-1}E(\tau).
\]
where
As a result, the application of Lemma \ref{lemma3.1} with $\phi(t)=\int_0^t\xi(s)\,\rd s$, $\si=q-1$ and $\om=K'$ implies the second inequality in \eqref{3.2}.

The proof of Theorem \ref{theorem3.1} is completed.


\section*{Acknowledgements}

The second author is supported by JSPS KAKENHI Grant Numbers JP22K13954, JP23KK0049 and Guangdong Basic and Applied Basic Research Foundation (No.\! 2025A1515012248).



\begin{thebibliography}{100}

\bibitem{10}
R.A. Adams, J.F. Fournier, Sobolev Spaces (2nd Ed.), Academic Press, New York, 2003.

\bibitem{25}
A.M. Al-Mahdi, M.M. Al-Gharabli, New general decay results in an infinite memory viscoelastic problem with nonlinear damping, Bound. Value Probl., 2019 (2019) 140.

\bibitem{20}
K. Bartkowski, P. G\'orka, One dimensional Klein--Gordon equation with logarithmic nonlinearities, J. Phys. A 41 (35) (2008) 355--201.

\bibitem{12}
F. Belhannache, M.M. Algharabli, S.A. Messaoud, Asymptotic stability for a viscoelastic equation with nonlinear damping and very general type of relaxation functions, J. Dyn. Control Syst., 26 (1) (2020) 45--67.

\bibitem{17}
I. Bialynicki-Birula, J. Mycielski, Wave equations with logarithmic nonlinearities, Bull. Acad. Pol. Sci. Cl 23 (4) (1975) 461--466.

\bibitem{18}
I. Bialynicki-Birula, J. Mycielski, Nonlinear wave mechanics, Ann. Physics 100 (12) (1976) 62--93

\bibitem{19}
T. Cazenave< A. Haraux, Equations d'\'evolution avec nonlinearity logarithmique, 2 (1) (1980) 21--51.

\bibitem{4}
H. Di, Y. Shang, Z. Song, Initial boundary value problem for a class of strongly damped semilinear wave equations with logarithmic nonlinearity, Nonlinear Anal. Real World Appl. 51 (2020) 102968.

\bibitem{8}
X. Fan, D. Zhao, On the spaces $L^{p(x)}(\Omega)$ and $W^{k,p(x)}(\Omega)$, J. Math. Anal. Appl. 263 (2) (2001) 424--446.

\bibitem{9}
X.-L. Fan, Q.-H. Zhang, Existence of solutions for $p(x)$-Laplacian Dirichlet problem, Nonlinear Anal. 52 (8) (2003) 1843--1852.

\bibitem{21}
P. G\'orka, Logarithmic Klein--Gordon equation, Acta Phys. Polon. B 40 (2009) 59--66.

\bibitem{1}
T.G. Ha, S.-H. Park, Blow-up phenomena for a viscoelastic wave equation with strong damping and logarithmic nonlinearity, Adv. Differential Equations 2020 (2020) 235.

\bibitem{3}
N. Irkil, E. Pi\c skin, P. Agarwal, Global existence and decay of solutions for a system of viscoelastic wave equations of Kirchhoff type with logarithmic nonlinearity, Math. Methods Appl. Sci. 45 (5) (2022) 2921--3948.

\bibitem{22}
K.-P. Jin, L. Jin, T.-J. Xiao, Stability of initial-boundary value problem for quasilinear viscoelastic equations, Electron. J. Differential Equations 2020 (80) (2020) 1--15.

\bibitem{7}
D. Lars, P. Harjulehto, P. H\"ast\"o, M. Ru\v zi\v cka, Lebesgue and Sobolev Spaces with Variable Exponents, Springer, Berlin 2011.

\bibitem{15}
F. Li, Q. Gao, Blow-up of solution for a nonlinear Petrovsky type equation with memory, Appl. Math. Comput. 274 (2016) 383--392.

\bibitem{23}
M. Liao, B. Guo, X. Zhu, Energy decay rates of solutions to a viscoelastic wave equation with variable exponents and weak damping, arXiv:2011.11185v1 (2020).

\bibitem{14}
M. Kafini, S. Messaoudi, Local existence and blow up of solutions to a logarithmic nonlinear wave equation with delay, Appl. Anal. 99 (3) (2020) 530--547.

\bibitem{16}
P. Martinez, A new method to obtain decay rate estimates for dissipative systems, ESAIM Control Optim. Calc. Var. 4 (1999) 419--444.

\bibitem{11}
M.I. Mustafa, Optimal decay rates for the viscoelastic wave equation, Math. Methods Appl. Sci. 41 (1) (2018) 192--204.

\bibitem{24}
S.A. Messaoudi, W. Al-Khulaifi, General and optimal decay for a quasilinear viscoelastic equation, Appl. Math. Lett. 66 (2017) 16--22.

\bibitem{2}
S.-H. Park, J.-R. Kang, Blow-up of solutions for a viscoelastic wave equation with variable exponents, Math. Methods Appl. Sci. 42 (6) (2019) 2083--2097.

\bibitem{q2}
Q. Peng, Y. Liu, Energy decay of viscoelastic equations with nonlinear damping and polynomial nonlinearity, Appl. Anal. 104 (18) (2025) 3497--3518.

\bibitem{PL26}
Q. Peng, Y. Liu, Exponential stability for an infinite memory wave equation with frictional damping and logarithmic nonlinear terms, Nonlinear Anal. Real World Appl. 88 (2026) 104470.

\bibitem{qing}
Q. Peng, Z. Zhang, Stabilization and blow-up in an infinite memory wave equation with logarithmic nonlinearity and acoustic boundary conditions, J. Syst. Sci. Complex. 37 (4) (2024) 1368--1391

\bibitem{qing2}
Q. Peng, Z. Zhang, Stabilization and blow-up for a class of weakly damped Kirchhoff plate equation with logarithmic nonlinearity, Indian J. Pure Appl. Math. 56 (2025) 711--727.

\bibitem{uk}
Y. Ueda, R. Duan, S. Kawashima, Decay structure for symmetric hyperbolic systems with non-symmetric relaxation and its applications, Arch. Rational Mech. Anal. 205 (2012) 239--266.

\end{thebibliography}
\end{document}